\documentclass[12pt,twoside]{amsart}
\usepackage{amsmath, amsthm, amscd, amsfonts, amssymb, graphicx, mathrsfs,dsfont, marvosym}
\newtheorem{thm}{Theorem}[section]
\newtheorem{cor}[thm]{Corollary}

\newtheorem{defn}[thm]{Definition}
\newtheorem{rem}[thm]{\bf{Remark}}

\numberwithin{equation}{section}
\newtheorem{ill}[thm]{ Illustration}
\newtheorem{exm}[thm]{Example}

\begin{document}
\begin{center}
	{\bf{An existence result of a functional integral equation via Darbo type theorem and an iterative algorithm to solve it
	}}
	\vspace{.5cm}
	\\ Bhuban Chandra Deuri$^{a,\ast}$ and Anupam Das$^{b}$
	\vspace{.2cm}
	
	$^{a,\ast}$Department of Mathematics, Rajiv Gandhi University, Rono Hills, Doimukh-791112, Arunachal Pradesh, India\\
	$^{b}$Department of Mathematics, Cotton University, Panbazar, Guwahati-781001,
	Assam, India\\

	$^{b,\ast}$Corresponding author. E-mail:bhuban.math@gmail.com\\
	Contributing author. Email: math.anupam@gmail.com 
\end{center}

\title{}
\author{}
\begin{abstract} 
In our study, Darbo's fixed point theorem and the coupled fixed point theorem have been extended. The two theorems were generalized using $C-class$ mapping. Using the Darbo type theorem, we gave an existence result for a functional integral equation along with an appropriate illustration. In addition, we explored the numerical solution of a given functional integral problem using modified homotopy perturbation.
	
	\vskip 0.5cm
	
	
	\vskip 0.5cm
	
	\textbf{MSC subject classification No:}  45G10;47H08;47G10;47H10
\end{abstract}

\maketitle
\pagestyle{myheadings}
\markboth{\rightline {\scriptsize Deuri}}
{\leftline{\scriptsize }}

\maketitle

\section{ Introduction}
The fixed point theorems of Schauder and Darbo are essential for solving functional integral problems. The concept of a $\mathscr{MNC}$ is important in fixed point theory.
This concept was first articulated in Kuratowski's \cite{od2} pioneering study. G. Darbo\cite{od5} came up with a theory in the middle of the 19th century that said fixed points are always present when condensing operators are used.

In the year 1980, J. Bana\'{s}\cite{Banas1} defined the measure of noncompactness and many researchers solved different problems such as integral equations, functional integral equations, differential equations, and fractional integral equations etc..

Fixed point theory and the $\mathscr{MNC}$ had numerous applications for solving many forms of fractional integral equations (see \cite{7Arab,7n7,7ad11,7vp,7b2,7b1,7b3}).
The solvability of in the functional integral equations have been discussed by many authors in Banach spaces (see \cite{71Arab,71n7,71ad11,71vp,71b2,71b1,71b3}).

Many researchers generalized Darbo's types theorem with applications.
 Das et al.\cite{7das1} established some new generalizations couple fixed point theorem using C-class functions. Inzamamul et al. \cite{7Inzamamul}  analyzed the concept of the solvability of implicit fractional order integral equation in $l_{p}(1\leq p<\infty)$ space via generalized Darbo type theorem.

In this paper, we have generalized two well-known theorems, namely, Darbo's fixed point theorem and the coupled fixed point theorem. Both theorems were generalized with the help of $C-class$ mapping. Herein, first we presented an existence result of a functional integral equation along with a suitable example using the Darbo type theorem. In addition to this, we also discussed the numerical solution of a given functional integral equation with the help of modified homotopy perturbation.

\par Let a real Banach space $\left( \mathscr{B}, \parallel . \parallel \right)$, also let 
$B(c,r)=\left\lbrace b \in \mathscr{B}:\parallel b-c \parallel \leq r \right\rbrace.$ If $\mathscr{K}(\neq \emptyset) \subseteq \mathscr{B}.$ Then

\begin{itemize}
	\item $\bar{\mathscr{K}}=$ the closure of $\mathscr{K},$
	\item $Conv\mathscr{K}=$ the convex closure of $\mathscr{K},$
	\item $\mathbb{R}=(-\infty, \infty),$ 
	\item $\mathbb{R}_{+}=\left[ 0, \infty \right) .$
	\item $\mathfrak{M}_{\mathscr{B}}=$ The collection of all nonempty and bounded subsets of $\mathscr{B},$
	\\
	also
	\item $\mathfrak{N}_{\mathscr{B}}=$ The collection of all relatively compact sets.
\end{itemize}

The axioms of a $\mathscr{MNC}$  is given below \cite{Banas1}.
\begin{defn}\label{dm}
	A mapping $\mu :\mathfrak{M}_{\mathscr{B}} \rightarrow \left[ 0, \infty\right)$ is an $\mathscr{MNC}$ in $\mathscr{B}$ if it is consistent with the following axioms:
	\begin{enumerate}
		\item[(i)] For all $\mathscr{K} \in \mathfrak{M}_{\mathscr{B}},$ we get
		$\mu(\mathscr{K})=0$ gives $\mathscr{K}$ is relatively compact.
		\item[(ii)]  ker $\mu = \left\lbrace \mathscr{K} \in \mathfrak{M}_{\mathscr{B}}: \mu\left( \mathscr{K}\right)=0   \right\rbrace \neq \emptyset$ and ker $\mu \subset \mathfrak{N}_{\mathscr{B}}.$
		\item[(iii)] $\mathscr{K} \subseteq \mathscr{K}_{1} \implies \mu\left( \mathscr{K}\right) \leq \mu\left( \mathscr{K}_{1}\right).$
		\item[(iv)] $\mu\left( \bar{\mathscr{K}}\right)=\mu\left( \mathscr{K}\right).$
		\item[(v)] $\mu\left( Conv \mathscr{K}\right)=\mu\left( \mathscr{K}\right).$
		\item[(vi)] $\mu\left( \varpi \mathscr{K} +\left(1- \varpi \right)\mathscr{K}_{1} \right) \leq \varpi \mu\left( \mathscr{K}\right)+\left(1- \varpi \right)\mu\left( \mathscr{K}_{1}\right)$ for $\varpi \in [ 0, 1 ).$
		\item[(vii)] if $\mathscr{K}_{q} \in \mathfrak{M}_{\mathscr{B}}, \:\mathscr{K}_{q}= \bar{\mathscr{K}}_{q}, \: 
		\mathscr{K}_{q+1} \subset \mathscr{K}_{q}$ for $q=1,2,...$ \;also\;\; $\lim\limits_{q \rightarrow \infty}\mu\left( \mathscr{K}_{q}\right)=0$ then $\mathscr{K}_{\infty}=\bigcap_{q=1}^{\infty}\mathscr{K}_{q} \neq \emptyset.$ 
	\end{enumerate}
\end{defn}
\par In $(vii),$ $\mathscr{K}_{\infty}$ is a member of the family $ker \mu .$\\
 Since $\mu(\mathscr{K}_{\infty}) \leq \mu(\mathscr{K}_{q}),$  \; $\mu(\mathscr{K}_{\infty})=0.$
 \\ So, $\mathscr{K}_{\infty}=\bigcap_{q=1}^{\infty}\mathscr{K}_{q} \in ker \mu.$ 

\subsection*{Some significant theorems and definitions}
\par The follows two essential theorems to consider:

\begin{thm}(Shauder \cite{sh})\label{Ts}
	If $\mathscr{G}$ is a nonempty, bounded, closed and convex subset$(\mathscr{NBCCS})$ of a Banach Space ${\mathscr{B}},$ then the mapping $\Psi: \mathscr{G} \rightarrow \mathscr{G}$ has at least one fixed point if it is continuous and compact.
\end{thm}

\begin{thm}(Darbo\cite{od5})\label{Td}
	Let $\mathscr{G}$ be a $\mathscr{NBCCS}$ of a Banach Space ${\mathscr{B}}.$ Suppose that $\Psi: \mathscr{G} \rightarrow \mathscr{G}$ is a continuous mapping such that 
	\[
	\mu(\Psi\mathscr{K})\leq \varpi\mu(\mathscr{K}),\; \mathscr{K}\subseteq \mathscr{G}.
	\]
	Where, $\varpi\in \left[ 0,1\right) $ is a constant. So, $\Psi$ has a fixed point.
\end{thm}
\par To prove a generalized of Darbo's type theorem, the following related definitions are necessary:

\begin{defn}\cite{7Inzamamul}\label{d7}
	Denote by $\mathcal{Q}$ the class of all $Q : \mathbb{R}_{+} \rightarrow \mathbb{R}$ with the below axioms: \\
	$(a)$ $Q$\; is\; non-decreasing.\\
	$(b)$ $Q$ \;is\; continuous.\\
	$(c)$ $\lim\limits_{k \rightarrow \infty}Q(t_k)=0\Leftrightarrow \lim\limits_{k \rightarrow \infty}t_k=0$ for each sequence $\left\lbrace t_k\right\rbrace \subset (0,\infty) .$\\
	That is  $Q(t) = 0$ iff $t = 0.$
\end{defn}

\begin{defn}\cite{7Inzamamul}\label{d71}
	Let $\Delta$ be the family of all function (known as comparison function) $\delta : \mathbb{R}_{+} \rightarrow \mathbb{R}_{+}$ with the below axioms:\\
	$(a)$ $\delta$ is monotone increasing.\\
	$(b)$ $\lim\limits_{q \rightarrow \infty}\delta^{q}(z)=0, $ for all $z>0$.
\end{defn}

\begin{defn}(\cite{7d72})\label{d72}
	Denote by $\mathbb{E}$ the class of all  $\mathscr{E}: \mathbb{R}_{+}\times\mathbb{R}_{+} \rightarrow \mathbb{R}$ with the following properties:\\
		$a$ $\mathscr{E}(z_1,z_2)\geq \max \left\lbrace z_1,z_2\right\rbrace $ for $z_1,z_2 \geq 0.$\\
		$b$ $\mathscr{E}$ is a continuous.\\
		$c$ $\mathscr{E}(z_1+z_2,x_1+x_2) \leq \mathscr{E}(z_1,x_1)+\mathscr{E}(z_2,x_2).$\\
	e.g $\mathscr{E}(z_1,z_2) = z_1+z_2.$
\end{defn}

\section{Darbo Fixed point Results}
\begin{thm}\label{T7}
	Suppose that $\mathscr{G}$ is a $\mathscr{NBCCS}$ of a Banach space $\mathscr{B}$ and, the function $\Psi:\mathscr{G} \rightarrow \mathscr{G}$ is a continuous  with
	\begin{equation}\label{eq7}
		\mu(\Psi \mathscr{K})>0\implies Q\left\lbrace \mathscr{E}\left( \mu\left(\Psi \mathscr{K}\right) ,F \left( \mu\left(\Psi \mathscr{K}\right) \right)\right)\right\rbrace   \leq 
		\delta[Q\left\lbrace \mathscr{E}\left( \mu\left(\mathscr{K}\right) ,F \left( \mu\left(\mathscr{K}\right) \right)\right)\right\rbrace ] 
	\end{equation}
	where $\mathscr{K} \subset \mathscr{G}$ also $\mu$ is an arbitrary $\mathscr{MNC}$  and $ Q \in \mathcal{Q}, \;\delta \in \Delta\; $ and $\mathscr{E}\in\mathbb{E}.$ Also $ F : \mathbb{R}_{+} \rightarrow \mathbb{R}_{+}$ is non-decreasing continuous function.
	So that $\Psi$ has atleast one fixed point in $\mathscr{G}.$
\end{thm}
\begin{proof}
	Assume that a sequence $\left\lbrace \mathscr{C}_{k}\right\rbrace _{k=1}^{\infty}$ with $\mathscr{C}_{0}=\mathscr{C}$
	and $\mathscr{C}_{k}=Conv(\Psi\mathscr{C}_{k-1})$ for $k\geq 1.$  We have, $\Psi\mathscr{C}_{0}=\Psi\mathscr{C}\subseteq\mathscr{C}_{0}=\mathscr{C}_{0},$ 
	$\mathscr{C}_{1}=Conv(\Psi\mathscr{C}_{0})\subseteq\mathscr{C}=\mathscr{C}_{0};  $
	therefore, by proceeding in the same manner gives 
	$\mathscr{C}_{0} \supseteq\mathscr{C}_{1} \supseteq \mathscr{C}_{2}\supseteq \mathscr{C}_{3}\supseteq \ldots \supseteq \mathscr{C}_{k}\supseteq \mathscr{C}_{k+1}\supseteq \ldots.$

   	\par If there exists an integer ${L}\geq 0$ such that $\mathscr{E}\left(\mu( \mathscr{C}_{L}),F(\mu( \mathscr{C}_{L}))\right) =0.$ which gives
   $\mu( \mathscr{C}_{L})=0.$ So $\mathscr{C}_{L}$ is compact set with $\Psi\mathscr{C}_{L}\subseteq Conv(\Psi\mathscr{C}_{L})=\mathscr{C}_{L+1}.$ Thus,  theorem \ref{Ts} (Schauder's) provides that
   $\Psi$ have a fixed point in $\mathscr{C}.$

	\par Moreover, if $\mathscr{E}\left(\mu( \mathscr{C}_{k}),F(\mu( \mathscr{C}_{k}))\right)>0,\; k \in \mathbb{N}.$
	\par By the equation \ref{eq7}, we now obtain 
	\begin{align*}
		& Q\left\lbrace  \mathscr{E}\left(\mu( \mathscr{C}_{k+1}),F(\mu( \mathscr{C}_{k+1}))\right)\right\rbrace \\
		& = Q\left\lbrace  \mathscr{E}\left(\mu( Conv(\Psi\mathscr{C}_{k})),F(\mu( Conv(\Psi\mathscr{C}_{k})))\right)\right\rbrace\\
		&= Q\left\lbrace  \mathscr{E}\left(\mu( \Psi\mathscr{C}_{k}),F(\mu( \Psi\mathscr{C}_{k}))\right)\right\rbrace\\
		& \leq \delta[Q\left\lbrace  \mathscr{E}\left(\mu( \mathscr{C}_{k}),F(\mu( \mathscr{C}_{k}))\right)\right\rbrace ]\\
		& \leq \delta[Q\left\lbrace  \mathscr{E}\left(\mu( Conv(\Psi\mathscr{C}_{k-1})),F(\mu( Conv(\Psi\mathscr{C}_{k-1})))\right)\right\rbrace ]\\
		& = \delta[Q\left\lbrace  \mathscr{E}\left(\mu( \Psi\mathscr{C}_{k-1}),F(\mu( \Psi\mathscr{C}_{k-1}))\right)\right\rbrace ]\\
		& \leq \delta^{2}[Q\left\lbrace  \mathscr{E}\left(\mu(\mathscr{C}_{k-1}),F(\mu( \mathscr{C}_{k-1}))\right)\right\rbrace ]\\
		& \dots\dots\\
		& \leq \delta^{k+1}[Q\left\lbrace  \mathscr{E}\left(\mu(\mathscr{C}_{0}),F(\mu( \mathscr{C}_{0}))\right)\right\rbrace ].
	\end{align*}
    Thus, we have
    \[
    Q\left\lbrace  \mathscr{E}\left(\mu( \mathscr{C}_{k+1}),F(\mu( \mathscr{C}_{k+1}))\right)\right\rbrace\leq \delta^{k+1}[Q\left\lbrace  \mathscr{E}\left(\mu(\mathscr{C}_{0}),F(\mu( \mathscr{C}_{0}))\right)\right\rbrace ].
    \]
    As $k \rightarrow \infty$ and applying definition \ref{d71}, we obtain
	\[\lim\limits_{k \rightarrow \infty} Q\left\lbrace  \mathscr{E}\left(\mu( \mathscr{C}_{k+1}),F(\mu( \mathscr{C}_{k+1}))\right)\right\rbrace=0.\]
	By using (\ref{d7}), we get
	\[\lim\limits_{k \rightarrow \infty}\mathscr{E}\left(\mu( \mathscr{C}_{k+1}),F(\mu( \mathscr{C}_{k+1}))\right) =0.\]
	Also, using by $(a)$ of definition \ref{d72}, we have
	\[\max\left\lbrace \mu( \mathscr{C}_{k+1}),F(\mu( \mathscr{C}_{k+1}))\right\rbrace\leq \mathscr{E}\left(\mu( \mathscr{C}_{k+1}),F(\mu( \mathscr{C}_{k+1}))\right).\]
	As $k \rightarrow \infty,$ this gives
	\[
	0\leq \max\left\lbrace \lim\limits_{k \rightarrow \infty}\mu( \mathscr{C}_{k+1}),\lim\limits_{k \rightarrow \infty}F(\mu( \mathscr{C}_{k+1}))\right\rbrace\leq 0.
	\]
	Since, $\mathbb{F}\geq 0,$ we get 
	\[
	\lim\limits_{k \rightarrow \infty}\mu( \mathscr{C}_{k+1})=0 \; and\;\lim\limits_{k \rightarrow \infty}F(\mu( \mathscr{C}_{k+1})) =0.
	\]
	Again, since $\mathscr{C}_{k}\supseteq \mathscr{C}_{k+1}$ and $\mathscr{C}_{k}\supseteq \Psi\mathscr{C}_{k}$ forall $k\geq 1.$\\
	 So from \ref{dm}, it follows that $\mathscr{C}_{\infty}=\bigcap_{k=1}^{\infty}\mathscr{C}_{k}$ is non-empty, convex,  closed set, invariant under the function $\Psi$ also belong to $ker\mu .$
	\\ Thus, from Schauder’s FPT (theorem \ref{Ts}) gives the required result. 
	Hence, the completed proof.
\end{proof}
\begin{thm}\label{T71}
Suppose that $\mathscr{G}$ is a $\mathscr{NBCCS}$ of a Banach space $\mathscr{B}$ and, the function $\Psi:\mathscr{G} \rightarrow \mathscr{G}$ is a continuous  with
\begin{equation}
	\mu(\Psi \mathscr{K})>0\implies Q\left\lbrace \mathscr{E}\left( \mu\left(\Psi \mathscr{K}\right) ,F \left( \mu\left(\Psi \mathscr{K}\right) \right)\right)\right\rbrace   \leq 
	\lambda Q\left\lbrace \mathscr{E}\left( \mu\left(\mathscr{K}\right) ,F \left( \mu\left(\mathscr{K}\right) \right)\right)\right\rbrace 
\end{equation}
where $\mathscr{K} \subset \mathscr{G}$ and $\mu$ is an arbitrary $\mathscr{MNC}$  and $ Q \in \mathcal{Q} $ and $\mathscr{E}\in\mathbb{E}.$ Also $ F : \mathbb{R}_{+} \rightarrow \mathbb{R}_{+}$ is non-decreasing continuous mapping.
So that $\Psi$ has atleast one fixed point in $\mathscr{G}.$
\end{thm}
\begin{proof}
	Taking $\delta(z)=\lambda z, \; 0<\lambda<1,\; z>0$ in the Theorem \; \ref{T7}. We achieve the aforementioned theorem.
\end{proof}
\begin{thm}\label{T72}
	Suppose that $\mathscr{G}$ is a $\mathscr{NBCCS}$ of a Banach space $\mathscr{B}$ and, the function $\Psi:\mathscr{G} \rightarrow \mathscr{G}$ is a continuous  with
	\begin{equation}
		\mu(\Psi \mathscr{K})>0\implies  \mathscr{E}\left( \mu\left(\Psi \mathscr{K}\right) ,F \left( \mu\left(\Psi \mathscr{K}\right) \right)\right)   \leq 
		\lambda \mathscr{E}\left( \mu\left(\mathscr{K}\right) ,F \left( \mu\left(\mathscr{K}\right) \right)\right)
	\end{equation}
	where $\mathscr{K} \subset \mathscr{G}$ and $\mu$ is an arbitrary $\mathscr{MNC}$  and $\mathscr{E}\in\mathbb{E}.$ Also $ F : \mathbb{R}_{+} \rightarrow \mathbb{R}_{+}$ is non-decreasing continuous mapping.
	So that $\Psi$ has atleast one fixed point in $\mathscr{G}.$
\end{thm}
\begin{proof}
	Taking $Q(z)=z$ in the theorem \ref{T71} we achieve , the theorem \ref{T72}.
\end{proof}

\begin{cor}\label{c7}
	Suppose that $\mathscr{G}$ is a $\mathscr{NBCCS}$ of a Banach space $\mathscr{B}$ and, the function $\Psi:\mathscr{G} \rightarrow \mathscr{G}$ is a continuous  with
	\begin{equation}
		\mu(\Psi \mathscr{K})>0\implies \mu(\Psi \mathscr{K}) +F \left( \mu\left(\Psi \mathscr{K}\right) \right)    \leq 
		\lambda [\mu(\mathscr{K}) +F \left( \mu\left(\mathscr{K}\right) \right)]
	\end{equation}
	where $\mathscr{K} \subset \mathscr{G}$ and $\mu$ is an arbitrary $\mathscr{MNC.}$ Also $ F : \mathbb{R}_{+} \rightarrow \mathbb{R}_{+}$ is non-decreasing continuous mapping.
	So that $\Psi$ has atleast one fixed point in $\mathscr{G}.$
\end{cor}
\begin{proof}
	Setting $\mathscr{E}(z_1,z_2) = z_1+z_2$ in the theorem \ref{T72}. We achieve the aforementioned corollary.
\end{proof}
\begin{cor}\label{c71}
	Suppose that $\mathscr{G}$ is a $\mathscr{NBCCS}$ of a Banach space $\mathscr{B}$ and, the function $\Psi:\mathscr{G} \rightarrow \mathscr{G}$ is a continuous  with
	\begin{equation}
		\mu(\Psi \mathscr{K})>0\implies \mu(\Psi \mathscr{K})   \leq 
		\lambda \mu( \mathscr{K}) 
	\end{equation}
	where $\mathscr{K} \subset \mathscr{G}$ and $\mu$ is an arbitrary $\mathscr{MNC.}$ 
	So that $\Psi$ has atleast one fixed point in $\mathscr{G}.$
\end{cor}
\begin{proof}
	 Putting $\mathbb{F}(z) =0, \; 0<\lambda<1$ in the Corollary \ref{c7}. We get, Darbo $FPT.$
\end{proof}

\section{Coupled fixed point Results}
\begin{defn}(\cite{7d73}) \label{d73}
	An element $(\mathscr{C},\mathscr{D} ) \in \mathscr{H} \times \mathscr{H}$ is called a coupled fixed point$(\mathscr{CFP})$ of a mapping $\Psi :\mathscr{H} \times \mathscr{H} \rightarrow \mathscr{H} $ if $\Psi (\mathscr{C},\mathscr{D} ) =  \mathscr{C}$ and $\Psi (\mathscr{D} ,\mathscr{C}) =  \mathscr{D}$.
\end{defn}
\begin{thm}(\cite{Banas1})\label{T73}
	Suppose $\mu_1, \mu_2,\dots , \mu_n$ is the $\mathscr{MNC}$  in $\mathscr{B}_1, \mathscr{B}_2,\dots , \mathscr{B}_n$	respectively. Moreover, suppose the function $\mathscr{L}: \mathbb{R}^{n}_{+} \rightarrow \mathbb{R}_{+}$ is convex also $\mathscr{J} (h_1,h_2,\dots , h_n) = 0$ $\Leftrightarrow$ $h_{t} = 0$ for $t = 1, 2,\dots,n$ then $\mu(\mathscr{L}) = \mathscr{J}(\mu_1(\mathscr{L}_1), \mu_2(\mathscr{L}_2), \dots , \mu_n(\mathscr{L}_n))$
	define a $\mathscr{MNC}$ in  $\mathscr{B}_1, \mathscr{B}_2,\dots , \mathscr{B}_n$, where $\mathscr{L}_t$ denote the natural projections of $\mathscr{L}$ into $\mathscr{B}_t$	for $t = 1, 2,3,\dots , n$.
\end{thm}
\begin{exm}(\cite{Banas1})
	Suppose $\mu$ is a $\mathscr{MNC}$ on the Banach space $\mathscr{B}$. Define $ \mathscr{J}(\mathscr{C},\mathscr{D} ) = \mathscr{C}+\mathscr{D} ; \mathscr{C},\mathscr{D} \in \mathbb{R}_{+}$. Then $ \mathscr{J}$ has all the propertiesmentioned in Theorem $\ref{T73}.$ Hence, $\tilde{\mu}(\mathcal{K}) = \mu(\mathscr{L}_1) + \mu(\mathscr{L}_2)$
	is a $\mathscr{MNC}$ in the space $\mathscr{B}\times \mathscr{B}$,\;\; where $\mathscr{L}_t$ , $ t = 1,\; 2$ denote the natural projections of $\mathscr{L}$.
\end{exm}
\begin{thm}\label{TT7}
	Suppose $\mathscr{C}$ is a $\mathscr{NBCCS}$ of a Banach Space $\mathscr{B}$ and also, let $\Psi: \mathscr{C}\times\mathscr{C} \rightarrow \mathscr{C}$ be a continuous operator with
	\begin{equation} \label{eq71}
		Q[\mathscr{E}\left\lbrace \mu(\Psi (s_1\times s_2)) ,F(\mu(\Psi (s_1\times s_2)))\right\rbrace ] \leq \frac{1}{2}\delta[Q[\mathscr{E}\left\lbrace \mu(s_1),F(\mu(s_1))\right\rbrace +\mathscr{E}\left\lbrace\mu(s_2) ,F(\mu(s_2))\right\rbrace ]]
	\end{equation}
	for all $s_1, s_2\subseteq \mathscr{C}$, where $Q$ , $\mathscr{E}$ and $F$ are as in theorem {\ref{T7}} also $\mu$ is an arbitrary $\mathscr{MNC}.$ In addition, we assume $Q(\mathscr{A}+\mathscr{S}) \leq Q(\mathscr{A})+Q(\mathscr{S});$ $ \mathscr{A},\mathscr{S} \geq 0$ and $F(\mathscr{A}+\mathscr{S}) \leq F(\mathscr{A})+F(\mathscr{S});$ $ \mathscr{A},\mathscr{S} \geq 0.$Then $\Psi$ has at least a couple fixed point in $\mathscr{C}$. 
\end{thm}
\begin{proof}
	Consider a mapping $\Psi^{cf}: \mathscr{C}\times\mathscr{C} \rightarrow \mathscr{C}\times\mathscr{C}$  by $\Psi^{cf}(\mathscr{A},\mathscr{S}) = (\Psi(\mathscr{A},\mathscr{S}) , \Psi(\mathscr{S},\mathscr{A}));$ $\mathscr{A},\mathscr{S}\in\mathscr{C}.$ It is trivial that $\Psi^{cf}$ is continuous.

	Let $s\subseteq \mathscr{C}\times\mathscr{C}$ be non-empty. We have $\tilde{\mu}(s) = \mu(s_1) + \mu(s_2)$ is an MNC, where $s_1, s_2$ are the natural projections of $s$ into $\mathscr{B}.$\\

	We get,
	\begin{align*}
		&Q[\mathscr{E}\left\lbrace \tilde{\mu}(\Psi^{cf} (s)) ,F(\tilde{\mu}(\Psi^{cf} (s)))\right\rbrace ]\\
		&\leqslant Q[\mathscr{E}\left\lbrace \tilde{\mu}(\Psi (s_1\times s_2)\times\Psi (s_2\times s_1)) ,F(\tilde{\mu}(\Psi (s_1\times s_2)\times\Psi (s_2\times s_1)))\right\rbrace ]\\
		&=Q[\mathscr{E}\left\lbrace \mu(\Psi (s_1\times s_2))+\mu(\Psi (s_2\times s_1)) ,F(\mu(\Psi (s_1\times s_2))+\mu(\Psi (s_2\times s_1)))\right\rbrace ]\\
		&\leq Q[\mathscr{E}\left\lbrace \mu(\Psi (s_1\times s_2))+\mu(\Psi (s_2\times s_1)) ,F(\mu(\Psi (s_1\times s_2)))+F(\mu(\Psi (s_2\times s_1)))\right\rbrace ] 
		\\
		&\leq Q[\mathscr{E}\left\lbrace \mu(\Psi (s_1\times s_2)),F(\mu(\Psi (s_1\times s_2)))\right\rbrace +\mathscr{E}\left\lbrace\mu(\Psi (s_2\times s_1)) ,F(\mu(\Psi (s_2\times s_1)))\right\rbrace ] 
		\\
		&\leq Q[\mathscr{E}\left\lbrace \mu(\Psi (s_1\times s_2)),F(\mu(\Psi (s_1\times s_2)))\right\rbrace ]+Q[\mathscr{E}\left\lbrace\mu(\Psi (s_2\times s_1)) ,F(\mu(\Psi (s_2\times s_1)))\right\rbrace ] 
		\\
		&\leq \frac{1}{2}\delta[Q[\mathscr{E}\left\lbrace \mu(s_1),F(\mu(s_1))\right\rbrace +\mathscr{E}\left\lbrace\mu(s_2) ,F(\mu(s_2))\right\rbrace ]]\\
		& +\frac{1}{2}\delta[Q[\mathscr{E}\left\lbrace \mu(s_2),F(\mu(s_2))\right\rbrace +\mathscr{E}\left\lbrace\mu(s_1) ,F(\mu(s_1))\right\rbrace ]]
		\\
		&\leq \delta[Q[\mathscr{E}\left\lbrace \mu(s_1),F(\mu(s_1))\right\rbrace +\mathscr{E}\left\lbrace\mu(s_2) ,F(\mu(s_2))\right\rbrace ]] 
		\\
		&= \delta[Q[\mathscr{E}\left\lbrace \mu(s_1)+\mu(s_2), F(\mu(s_1)) +F(\mu(s_2))\right\rbrace ]] 
		\\
		&= \delta[Q[\mathscr{E}\left\lbrace \tilde{\mu}(s),F(\tilde{\mu}(s))\right\rbrace ]
	\end{align*}
	By Theorem {\ref{T7}}, we conclude that $\Psi^{cf}$ has minimum of one fixed point in $\mathscr{C}\times\mathscr{C}.$ That is, $\Psi$ has
	minimum of one coupled fixed point.
\end{proof}

\section{Application}
\subsection{$\mathscr{MNC}$ on $\Bar{\mathds{C}}([0,\mathbb{T}])$:}
Suppose  $\mathscr{B}=\Bar{\mathds{C}}(\Bar{\mathds{J}})$ is the space of all real valued continuous function defined on $\Bar{\mathds{J}},$ where $\Bar{\mathds{J}}=[0,\mathbb{T}].$ Then,
\[
\parallel \mathscr{T} \parallel=\sup\left\lbrace \left| \mathscr{T}(z)\right|:z \in \Bar{\mathds{J}} \right\rbrace ,\; \mathscr{T}\in \mathscr{B}.
\]

Let $\mathscr{F}(\neq \phi) \subseteq \mathscr{B}$ be   bounded. For  $\delta >0$ and $\mathscr{T} \in \mathscr{F}$ such that modulus of the continuity of $\mathscr{T}$ is defined as
\[\mu(\mathscr{T},\delta)=\sup \left\lbrace \left| \mathscr{T}(z_{1})-\mathscr{T}(z_{2})\right|: z_{1},z_{2} \in \Bar{\mathds{J}} , \left| z_{1}-z_{2}\right|\leq \delta \right\rbrace.
\]
Moreover, we set
\[ \mu(\mathscr{F},\delta)=\sup\left\lbrace \mu(\mathscr{T},\delta): \mathscr{T} \in \mathscr{F}\right\rbrace;~ \mu_{0}(\mathscr{F})=\lim\limits_{\delta \rightarrow 0}\mu(\mathscr{F},\delta).
\]
Here, the mapping $\mu_{0}$ is known as $\mathscr{MNC}$ in $\mathscr{B}$, with $\mathscr{Q}(\mathscr{F})=\frac{1}{2}\mu_{0}(\mathscr{F})$
(see \cite{Banas1}) as the Hausdorff $\mathscr{MNC}$ $\mathscr{Q}$.

\subsection{Existence results of functional integral equation :}
Let $-\infty\leq a<b<\infty,$ $\;0<\rho,k,\gamma<1$ and also let $\phi$ be a continuous function. Then, defines the integral $(k,\rho)$-fractional Hilfer left integral is given by (we refer \cite{7H1,7H2,7H3}).
\[
(^{\rho}_{k}\mathscr{J}^{\gamma}_{a+}\phi)(x)=\frac{\rho^{1-\frac{\gamma}{k}}}{k\Gamma_{k}(\gamma)}\int_{a}^{x}\frac{t^{\rho-1}}{\left(x^{\rho}-t^{\rho} \right)^{1-\frac{\gamma}{k}}}\phi(t)dt ,\;\;\;x>a\;\;x \in [a,b].
\]
Given that $a=1$ and $b=\mathbb{T}$, we compute
\[
(^{\rho}_{k}\mathscr{J}^{\gamma}\phi)(x)=\frac{\rho^{1-\frac{\gamma}{k}}}{k\Gamma_{k}(\gamma)}\int_{1}^{x}\frac{t^{\rho-1}}{\left(x^{\rho}-t^{\rho} \right)^{1-\frac{\gamma}{k}}}\phi(t)dt ,\;\;\;x>a\;\;x \in [a,b].
\]

In this study, we examine the subsequent functional integral equation:

\begin{equation}\label{eq75}
	\mathscr{L}=\varDelta(x,\mathscr{H}\left(x, \mathscr{L}(x)),(^{\rho}_{k}\mathscr{J}^{\gamma}\mathscr{L})(x) \right) ,
\end{equation}
where $x \in \Bar{\mathds{J}}=[1,\mathbb{T}]\;\;and\;\;0<\rho,k,\gamma<1.$\\
Let  define 
\[
\mathscr{Q}_{r_{0}}=\left\lbrace \mathscr{L}\in \mathscr{B}: \parallel \mathscr{L} \parallel \leq r_{0} \right\rbrace. \]
Assume that
\begin{enumerate}
	\item [(H1)]
	$\varDelta:\Bar{\mathds{J}} \times \mathbb{R}^{2} \rightarrow \mathbb{R},\;\; \mathscr{H}: \Bar{\mathds{J}} \times \mathbb{R}\rightarrow \mathbb{R}$ be continuous and $\exists$ constants $c_{1},\; c_{2},\; c_{3} \geq 0$ satisfying
	\[
	\left| \varDelta(x,\mathscr{H},\mathscr{J}_{1})-\varDelta(x,\bar{\mathscr{H}},\bar{\mathscr{J}}_{1})\right|\leq c_{1} \left|\mathscr{H}-\bar{\mathscr{H}} \right|+c_{2}\left|\mathscr{J}_{1}-\bar{\mathscr{J}}_{1} \right|, \;\;\; x\in \Bar{\mathds{J}};\; \mathscr{H},\mathscr{J}_{1},\bar{\mathscr{H}},\bar{\mathscr{J}}_{1} \in\mathbb{R}
	\]
	and
	\[
	\left|\mathscr{H}(x, \mathscr{L}_{1})-\mathscr{H}(x, \mathscr{L}_{2}) \right|\leq c_{3}\left|\mathscr{L}_{1}-\mathscr{L}_{2} \right|\;\;\;\;\;(\mathscr{L}_{1},\mathscr{L}_{2}\in \mathbb{R}).
	\]
	\item[(H2)] There exists  $r_{0}>0$ satisfying
	\[
	\bar{\varDelta}=\sup\left\lbrace \left| \varDelta(x,\mathscr{H},\mathscr{J}_{1})\right|:x \in \Bar{\mathds{J}},\mathscr{H}\in [-\hat{\mathscr{H}},\hat{\mathscr{H}}],\mathscr{J}_{1}\in [-\hat{\mathscr{J}},\hat{\mathscr{J}}]   \right\rbrace \leq r_{0}
	\]
	and \[c_{1}c_{3}<1,
	\]
	where \[\hat{\mathscr{H}}=\sup\left\lbrace \left|\mathscr{H} \right|: x \in \Bar{\mathds{J}}, \mathscr{L}(x)\in [-r_{0},r_{0}] \right\rbrace \]
	and
	\[
	\hat{\mathscr{J}}=\sup\left\lbrace \left|^{\rho}_{k}\mathscr{J}^{\gamma}\mathscr{L}(x) \right|: x \in \Bar{\mathds{J}}, \mathscr{L}(x)\in [-r_{0},r_{0}] \right\rbrace.
	\]
	\item[(H3)] $Max\left|\varDelta\left(x,0,0 \right)\right|=\mathscr{M}$ and ${\mathscr{H}}(x,0)=0$ for all $x \in \Bar{\mathds{J}}=[1,\mathbb{T}]$.
	\item[(H4)] The given below inequality has a positive result $r_{0}$ such that
	\begin{equation*}
		c_{1}c_{3}r_{0}+\frac{c_{2}r_{0}\rho^{-\frac{\gamma}{k}}}{\gamma \Gamma_{k}(\gamma)}\left( \mathbb{T}^{\rho}-1\right)^{\frac{\gamma}{k}}+\mathscr{M} \leq r_{0}.
	\end{equation*}
\end{enumerate}

\begin{thm}\label{T74}
	If the assumptions $(H1)-(H4)$ holds, then the equation (\ref{eq75}) has at least one solution in $\mathscr{B}=\Bar{\mathds{C}}(\Bar{\mathds{J}}).$
\end{thm}
\begin{proof}
	We take the operator $\mathscr{D}: \mathscr{B} \rightarrow \mathscr{B},$  defined by:
	\[(\mathscr{D} \mathscr{L})(x)= \varDelta(x,\mathscr{H}\left(x, \mathscr{L}(x)),(^{\rho}_{k}\mathscr{J}^{\gamma}\mathscr{L})(x) \right)
	\]

	{\bf Step 1:}   We show that $\mathscr{D}$
	maps $\mathscr{Q}_{r_0}$ into $\mathscr{Q}_{r_0}$.  Let  $\mathscr{L} \in \mathscr{Q}_{r_{0}}$. We now have
	\begin{eqnarray*}
		\left| (\mathscr{D} \mathscr{L})(x)\right|&\leq&\left|\varDelta(x,\mathscr{H}\left(x, \mathscr{L}(x)),(^{\rho}_{k}\mathscr{J}^{\gamma}\mathscr{L})(x) \right)-\varDelta\left(x,0,0 \right)\right|+\left|\varDelta\left(x,0,0 \right)\right|\\
		& \leq& c_{1}\left| \mathscr{H}(x, \mathscr{L}(x))-0\right| + c_{2}\left|(^{\rho}_{k}\mathscr{J}^{\gamma}\mathscr{L})(x)-0\right|+\left|\varDelta\left(x,0,0 \right)\right|\\
		& \leq& c_{1}c_{3}\left| \mathscr{L}(x)\right| + c_{2}\left|(^{\rho}_{k}\mathscr{J}^{\gamma}\mathscr{L})(x)\right|+\mathscr{M}.
	\end{eqnarray*}
	Also,
	
	\begin{eqnarray*}
		\left|(^{\rho}_{k}\mathscr{J}^{\gamma}\mathscr{L})(x)\right|&=&\left|\frac{\rho^{1-\frac{\gamma}{k}}}{k\Gamma_{k}(\gamma)}\int_{1}^{x}\frac{t^{\rho-1}}{\left(x^{\rho}-t^{\rho} \right)^{1-\frac{\gamma}{k}}}\mathscr{L}(t)dt \right|\\
		&\leq &\frac{\rho^{1-\frac{\gamma}{k}}}{k\Gamma_{k}(\gamma)}\int_{1}^{x}\frac{t^{\rho-1}}{\left(x^{\rho}-t^{\rho} \right)^{1-\frac{\gamma}{k}}}\left| \mathscr{L}(t)\right| dt \\
		&\leq &\frac{r_{0}\rho^{1-\frac{\gamma}{k}}}{k\Gamma_{k}(\gamma)}\int_{1}^{x}\frac{t^{\rho-1}}{\left(x^{\rho}-t^{\rho} \right)^{1-\frac{\gamma}{k}}}dt \\
		&\leq &\frac{r_{0}\rho^{-\frac{\gamma}{k}}}{\gamma \Gamma_{k}(\gamma)}\left( \mathbb{T}^{\rho}-1\right)^{\frac{\gamma}{k}}.
	\end{eqnarray*}
	
	Hence, $\parallel \mathscr{L}  \parallel <r_0$ gives
	\[
	\parallel \mathscr{D} \mathscr{L}\parallel\leq  c_{1}c_{3}r_{0}+\frac{c_{2}r_{0}\rho^{-\frac{\gamma}{k}}}{\gamma \Gamma_{k}(\gamma)}\left( \mathbb{T}^{\rho}-1\right)^{\frac{\gamma}{k}}+\mathscr{M}\leq r_{0}.
	\]
	
	It folloews from $(H4)$ that $\mathscr{D}$
	maps $\mathscr{Q}_{r_0}$ into $\mathscr{Q}_{r_0}$.\\

	{\bf Step 2:} Next we claim that, $\mathscr{D}$ is continuous on $\mathscr{Q}_{r_{0}}.$

	Fix $\delta >0,$ also let $\mathscr{L}, \bar{\mathscr{L}} \in \mathscr{Q}_{r_{0}}$ such that $\parallel \mathscr{L}- \bar{\mathscr{L}} \parallel < \delta.$ We now obtain
	\begin{eqnarray*}
		\left|  (\mathscr{D} \mathscr{L})(x)-  (\mathscr{D} \bar{\mathscr{L}})(x)\right|& \leq&  \left|\varDelta(x,\mathscr{H}\left(x, \mathscr{L}(x)),(^{\rho}_{k}\mathscr{J}^{\gamma}\mathscr{L})(x) \right)-\varDelta(x,\mathscr{H}\left(x, \bar{\mathscr{L}}(x)),(^{\rho}_{k}\mathscr{J}^{\gamma}\bar{\mathscr{L}})(x) \right)\right|\\
		& \leq& c_{1}\left| \mathscr{H}(x, \mathscr{L}(x))-\mathscr{H}(x, \bar{\mathscr{L}}(x))\right|+ c_{2}\left|(^{\rho}_{k}\mathscr{J}^{\gamma}\mathscr{L})(x)-(^{\rho}_{k}\mathscr{J}^{\gamma}\bar{\mathscr{L}})(x)\right| .
	\end{eqnarray*}
	Moreover,
	\begin{eqnarray*}
		\left|(^{\rho}_{k}\mathscr{J}^{\gamma}\mathscr{L})(x)-(^{\rho}_{k}\mathscr{J}^{\gamma}\bar{\mathscr{L}})(x)\right|
		&=&\left|\frac{\rho^{1-\frac{\gamma}{k}}}{k\Gamma_{k}(\gamma)}\int_{1}^{x}\frac{t^{\rho-1}}{\left(x^{\rho}-t^{\rho} \right)^{1-\frac{\gamma}{k}}}\left( \mathscr{L}(t)-\bar{\mathscr{L}}(t)\right) dt \right|\\
		&\leq &\frac{\rho^{1-\frac{\gamma}{k}}}{k\Gamma_{k}(\gamma)}\int_{1}^{x}\frac{t^{\rho-1}}{\left(x^{\rho}-t^{\rho} \right)^{1-\frac{\gamma}{k}}}\left|  \mathscr{L}(t)-\bar{\mathscr{L}}(t)\right|  dt \\
		&< &\frac{\delta\rho^{1-\frac{\gamma}{k}}}{k\Gamma_{k}(\gamma)}\int_{1}^{x}\frac{t^{\rho-1}}{\left(x^{\rho}-t^{\rho} \right)^{1-\frac{\gamma}{k}}}  dt \\
		&< &\frac{\delta\rho^{-\frac{\gamma}{k}}}{\gamma \Gamma_{k}(\gamma)}\left( \mathbb{T}^{\rho}-1\right)^{\frac{\gamma}{k}}.
	\end{eqnarray*}
	
	Hence, $\parallel \mathscr{L}- \bar{\mathscr{L}} \parallel < \delta$ gives
	\[
	\left|  (\mathscr{D} \mathscr{L})(x)-  (\mathscr{D} \bar{\mathscr{L}})(x)\right|<c_{1}c_{3}\delta+\frac{c_{2}\delta\rho^{-\frac{\gamma}{k}}}{\gamma \Gamma_{k}(\gamma)}\left( \mathbb{T}^{\rho}-1\right)^{\frac{\gamma}{k}}.
	\]

	As $\delta \rightarrow 0$, we get
	$$\left|  (\mathscr{D} \mathscr{L})(x)-  (\mathscr{D} \bar{\mathscr{L}})(x)\right| \rightarrow 0.$$
	\\ which indicates that $\mathscr{D}$ is continuous on $\mathscr{Q}_{r_{0}}.$\\

	{\bf Step 3:} Here, we estimate $\mathscr{D}$ in relation to $\mu_{0}.$ 
	 Consider that $\varpi (\neq \emptyset) \subseteq \mathscr{Q}_{r_{0}}.$  Let  $\delta >0$ be arbitrary, Also, we are now choose $\mathscr{L} \in \varpi$ with $x_{1}, x_{2} \in \Bar{\mathds{J}}$ such that $\left| x_{2}-x_{1}\right|\leq \delta $ and $x_{2} \geq x_{1}.$
	\par Now,
	\begin{align*}
		&\left| (\mathscr{D} \mathscr{L})(x_{2})-  (\mathscr{D} \mathscr{L})(x_{1})\right|\\
		& = \left| \varDelta(x_{2},\mathscr{H}\left(x_{2}, \mathscr{L}(x_{2})),(^{\rho}_{k}\mathscr{J}^{\gamma}\mathscr{L})(x_{2}) \right)-\varDelta(x_{1},\mathscr{H}\left(x_{1}, \bar{\mathscr{L}}(x_{1})),(^{\rho}_{k}\mathscr{J}^{\gamma}\mathscr{L})(x_{1}) \right)\right|\\
		&\leq  \left| \varDelta(x_{2},\mathscr{H}\left(x_{2}, \mathscr{L}(x_{2})),(^{\rho}_{k}\mathscr{J}^{\gamma}\mathscr{L})(x_{2}) \right)-\varDelta(x_{2},\mathscr{H}\left(x_{2}, \bar{\mathscr{L}}(x_{2})),(^{\rho}_{k}\mathscr{J}^{\gamma}\mathscr{L})(x_{1}) \right)\right|\\
		&+\left|\varDelta(x_{2},\mathscr{H}\left(x_{2}, \mathscr{L}(x_{2})),(^{\rho}_{k}\mathscr{J}^{\gamma}\mathscr{L})(x_{1}) \right)-\varDelta(x_{2},\mathscr{H}\left(x_{1}, \bar{\mathscr{L}}(x_{1})),(^{\rho}_{k}\mathscr{J}^{\gamma}\mathscr{L})(x_{1}) \right)\right|\\
		&+\left|\varDelta(x_{2},\mathscr{H}\left(x_{1}, \mathscr{L}(x_{1})),(^{\rho}_{k}\mathscr{J}^{\gamma}\mathscr{L})(x_{1}) \right)-\varDelta(x_{1},\mathscr{H}\left(x_{1}, \mathscr{L}(x_{1})),(^{\rho}_{k}\mathscr{J}^{\gamma}\bar{\mathscr{L}})(x_{1}) \right) \right|\\
		& \leq c_{2}\left|(^{\rho}_{k}\mathscr{J}^{\gamma}\mathscr{L})(x_{2})- (^{\rho}_{k}\mathscr{J}^{\gamma}\mathscr{L})(x_{1})\right|
		+c_{1}\left|\mathscr{H}(x_{2}, \mathscr{L}(x_{2}))-\mathscr{H}(x_{1}, \mathscr{L}(x_{1}))  \right|\\
		&+\mu_{\varDelta}(\Bar{\mathds{J}}, \delta)\\
		&\leq c_{2}\left|(^{\rho}_{k}\mathscr{J}^{\gamma}\mathscr{L})(x_{2})- (^{\rho}_{k}\mathscr{J}^{\gamma}\mathscr{L})(x_{1})\right|+c_{1}c_{3}\left| \mathscr{L}(x_{2})- \mathscr{L}(x_{1}) \right|+\mu_{\varDelta}(\Bar{\mathds{J}}, \delta),
	\end{align*}
	where
	\[
	\mu_{\varDelta}(\Bar{\mathds{J}}, \delta)=\sup\left\{\begin{array}{ccl}
		\left|\varDelta(x_{2},\mathscr{H},\mathscr{J}_{1})-\varDelta(x_{1},\mathscr{H},\mathscr{J}_{1})\right|:\left| x_{2}-x_{1}\right|\leq \delta;\\ x_{1},x_{2} \in \Bar{\mathds{J}},\mathscr{H}\in [-\hat{\mathscr{H}},\hat{\mathscr{H}}],\mathscr{J}_{1}\in [-\hat{\mathscr{J}},\hat{\mathscr{J}}] 
	\end{array}\right\}.
	\]
	Also,
	 \begin{align*}
		&\left|(^{\rho}_{k}\mathscr{J}^{\gamma}\mathscr{L})(x_{2})- (^{\rho}_{k}\mathscr{J}^{\gamma}\mathscr{L})(x_{1})\right|\\
		& = \left| \frac{\rho^{1-\frac{\gamma}{k}}}{k\Gamma_{k}(\gamma)}\int_{1}^{x_{2}}\frac{t^{\rho-1}}{\left(x_{2}^{\rho}-t^{\rho} \right)^{1-\frac{\gamma}{k}}}\mathscr{L}(t)dt
		-\frac{\rho^{1-\frac{\gamma}{k}}}{k\Gamma_{k}(\gamma)}\int_{1}^{x_{1}}\frac{t^{\rho-1}}{\left(x_{1}^{\rho}-t^{\rho} \right)^{1-\frac{\gamma}{k}}}\mathscr{L}(t)dt  \right|\\
		& \leq \frac{\rho^{1-\frac{\gamma}{k}}}{k\Gamma_{k}(\gamma)}\left| \int_{1}^{x_{2}}\frac{t^{\rho-1}}{\left(x_{2}^{\rho}-t^{\rho} \right)^{1-\frac{\gamma}{k}}}\mathscr{L}(t)dt
		-\int_{1}^{x_{1}}\frac{t^{\rho-1}}{\left(x_{1}^{\rho}-t^{\rho} \right)^{1-\frac{\gamma}{k}}}\mathscr{L}(t)dt  \right|\\
		& \leq \frac{\rho^{1-\frac{\gamma}{k}}}{k\Gamma_{k}(\gamma)}\left| \int_{1}^{x_{2}}\frac{t^{\rho-1}}{\left(x_{2}^{\rho}-t^{\rho} \right)^{1-\frac{\gamma}{k}}}\mathscr{L}(t)dt
		-\int_{1}^{x_{1}}\frac{t^{\rho-1}}{\left(x_{2}^{\rho}-t^{\rho} \right)^{1-\frac{\gamma}{k}}}\mathscr{L}(t)dt  \right|\\
		& + \frac{\rho^{1-\frac{\gamma}{k}}}{k\Gamma_{k}(\gamma)}\left| \int_{1}^{x_{1}}\frac{t^{\rho-1}}{\left(x_{2}^{\rho}-t^{\rho} \right)^{1-\frac{\gamma}{k}}}\mathscr{L}(t)dt
		-\int_{1}^{x_{1}}\frac{t^{\rho-1}}{\left(x_{1}^{\rho}-t^{\rho} \right)^{1-\frac{\gamma}{k}}}\mathscr{L}(t)dt  \right|\\
		& \leq \frac{\rho^{1-\frac{\gamma}{k}}\parallel \mathscr{L} \parallel}{k\Gamma_{k}(\gamma)}\left| \int_{x_{1}}^{x_{2}}\frac{t^{\rho-1}}{\left(x_{2}^{\rho}-t^{\rho} \right)^{1-\frac{\gamma}{k}}}dt
		+\int_{1}^{x_{1}}\frac{t^{\rho-1}}{\left(x_{2}^{\rho}-t^{\rho} \right)^{1-\frac{\gamma}{k}}}dt
		-\int_{1}^{x_{1}}\frac{t^{\rho-1}}{\left(x_{1}^{\rho}-t^{\rho} \right)^{1-\frac{\gamma}{k}}}dt  \right|\\
		& \leq \frac{\rho^{-\frac{\gamma}{k}}\parallel \mathscr{L} \parallel}{\gamma\Gamma_{k}(\gamma)}\left[ 2(x_{2}^{\rho}-x_{1}^{\rho})^{\frac{\gamma}{k}}+(x_{2}^{\rho}-1)^{\frac{\gamma}{k}}-(x_{1}^{\rho}-1)^{\frac{\gamma}{k}}\right].
	\end{align*}
As $\delta \rightarrow 0$, then $x_{2} \rightarrow x_{1}$ and so
$$\left|(^{\rho}_{k}\mathscr{J}^{\gamma}\mathscr{L})(x_{2})- (^{\rho}_{k}\mathscr{J}^{\gamma}\mathscr{L})(x_{1})\right|\rightarrow 0.$$
Hence
\begin{eqnarray*}
	\left| (\mathscr{D} \mathscr{L})(x_{2})-  (\mathscr{D} \mathscr{L})(x_{1}) \right|& \leq &c_{2}\left|(^{\rho}_{k}\mathscr{J}^{\gamma}\mathscr{L})(x_{2})- (^{\rho}_{k}\mathscr{J}^{\gamma}\mathscr{L})(x_{1})\right|+c_{1}c_{3}\left| \mathscr{L}(x_{2})- \mathscr{L}(x_{1}) \right|+\mu_{\varDelta}(\Bar{\mathds{J}}, \delta)
\end{eqnarray*}
which yields
\[
\mu(\mathscr{D} \mathscr{L},\delta)\leq c_{2}\left|(^{\rho}_{k}\mathscr{J}^{\gamma}\mathscr{L})(x_{2})- (^{\rho}_{k}\mathscr{J}^{\gamma}\mathscr{L})(x_{1})\right|+c_{1}c_{3}\left| \mathscr{L}(x_{2})- \mathscr{L}(x_{1}) \right|+\mu_{\varDelta}(\Bar{\mathds{J}}, \delta).
\]
	\par It is derived from the uniform continuity of $\varDelta$ on $\Bar{\mathds{J}} \times [-\hat{\mathscr{H}},\hat{\mathscr{H}}]\times [-\hat{\mathscr{J}},\hat{\mathscr{J}}] $  that \\ $\lim\limits_{\delta\rightarrow \;0}\mu_{\varDelta}(\Bar{\mathds{J}}, \delta)\rightarrow\; 0,$ as $\;\delta\; \rightarrow \;0$.
	\par Setting\; $\sup_{\mathscr{L} \in \varpi}$\; as well as $\;\delta \rightarrow 0\;$, we\; obtain
	\[
	\mu_{0}(\mathscr{D} \varpi)\leq c_{1}c_{3} \mu_{0}(\varpi).
	\]
	Hereby Corollary \;\ref{c71},   $\mathscr{D}\;$ has\; a \;fixed\; point\; in $\varpi \subseteq \mathscr{Q}_{r_{0}}.$ 
	\\Which\; implies\; the\; equation \;(\ref{eq75}) \;have\; a\; solution\; in\; $\mathscr{B}.$
\end{proof}

\begin{ill}\label{ex7}{\rm
		Take into consideration the following functional integral equation:
		\begin{equation}\label{eq76}
			\mathscr{L}(x)=\frac{\mathscr{L}(x)}{10+x^{4}}+\frac{^{\frac{1}{3}}_{\frac{1}{3}}
				\mathscr{J}^{\frac{2}{3}}
				\mathscr{L}(x)}{20}+ g(x)
		\end{equation}
		for $x\in [1,3]=\Bar{\mathds{J}}.$}\end{ill}
\noindent  Here,\\
\[g(x)=k^{2}x^{-\frac{\gamma}{\rho}}\frac{9+x^{\frac{\gamma}{\rho}}}{10+x^{\frac{\gamma}{\rho}}}=
\frac{1}{9}x^{-2}\frac{9+x^{2}}{10+x^{2}}\] \\and
\[
^{\frac{1}{3}}_{\frac{1}{3}}\mathscr{J}^{\frac{2}{3}}\mathscr{L}(x)=\frac{3^{2}}{\Gamma_{1/3}(2/3)}
\int_{1}^{x}\frac{\left(x^{1/3}-t^{1/3} \right)}{t^{2/3}}\mathscr{L}(t)dt.
\]
Also,
$$\varDelta(x,\mathscr{H},\mathscr{J}_{1})=\frac{1}{9}x^{-2}\frac{9+x^{2}}{10+x^{2}}+\mathscr{H}+\frac{\mathscr{J}_{1}}{20}$$
and
$$\mathscr{H}(x, \mathscr{L})=\frac{\mathscr{L}}{10+x^{4}}.$$
It is trivial that both $\varDelta$ and $\mathscr{H}$ are continuous satisfying
\[
\left| \mathscr{H}(x, \mathscr{L}_{1})-\mathscr{H}(x, \mathscr{L}_{2})\right| \leq \frac{\left|\mathscr{L}_{1}-\mathscr{L}_{2} \right| }{10}
\]
and
\[
\left|\varDelta(x,\mathscr{H},\mathscr{J}_{1})-\varDelta(x,\bar{\mathscr{H}},\bar{\mathscr{J}}_{1}) \right| \leq \left| \mathscr{H}-\bar{\mathscr{H}}\right|  +\frac{1}{20}\left|\mathscr{J}_{1}- \bar{\mathscr{J}}_{1}\right|,
\]
respectively. Therefore
$$c_{1}=1,\;  c_{2}=\frac{1}{20},\; c_{3}=\frac{1}{10}\;\;\;and\;\;\;c_{1}c_{3}=\frac{1}{10}<1.$$

If $\|\mathscr{L} \|\leq r_{0}$, then
\[
\hat{\mathscr{H}}=\frac{r_{0}}{10}
\]
and
\[
\hat{\mathscr{J}}=\frac{(1/3)^{-2}\left(  3^{\frac{1}{3}}-1\right)^{2}r_{0}}{(2/3)\Gamma_{1/3}(2/3)}.
\]
Further,
\[
\left| \varDelta(x,\mathscr{H},\mathscr{J}_{1})\right|\leq \frac{r_{0}}{10}+\frac{1}{20}.\frac{(1/3)^{-2}\left(  3^{\frac{1}{3}}-1\right)^{2}r_{0}}{(2/3)\Gamma_{1/3}(2/3)}+\frac{10}{99}\leq r_{0}.
\]
If we choose $r_{0}=3$ then
\[
\hat{\mathscr{H}}=\frac{3}{10},\;\hat{\mathscr{J}}=\frac{81\left(  3^{\frac{1}{3}}-1\right)^{2}}{40\Gamma_{1/3}(2/3)},  \mathscr{M}=\frac{10}{99}.
\]
which gives
\[
\bar{\varDelta}=0.5657\leq 3.
\]
However, assumption $(H4)$ is also satisfied for $r_{0}=3.$

Consequently, we have achieved all of the assumptions from $(H1)$ to $(H4)$ in the Theorem \ref{T74} .\\
From Theorem \ref{T74} , we can say that The equation (\ref{eq76}) have solutions in $\mathscr{B}=\Bar{\mathds{C}}(\Bar{\mathds{J}}).$

\section{
Establishing a convergent iterative algorithm by use of a MPH approach in order to solve an equation(\ref{eq76}):}
\noindent Theorem \ref{T74} proves the existence of a solution to the funcctional integral equation (\ref{eq76}). In this part, a significant convergent iterative technique is developed to approximate the solution of equation (\ref{eq76}) using the MHP approach. Homotopy and some of its modifications are used to solve nonlinear problems, nonlinear differential equations, integral equation systems, partial differential equations,\; infinite systems of nonlinear singular integral equations,\;   nonlinear integral equations, nonlinear functional integral equations\; and nonlinear quadratic integral equations, respectively (we refer \cite{MHP1, MHP2, MHP3, MHP4, MHP5, MHP6, MHP7,MHP8}).

\noindent Now, consider the following general form of (\ref{eq76}):
\begin{equation}\label{eq77}
	\aleph( \mathscr{L}(x)) - g(x) = 0,~~x\in \Bar{\mathds{J}}=[1,3],
\end{equation}
where $ \aleph$ is a fractional operator that is nonlinear and $g$ is a known function. 
By transforming operator $ \aleph $ into two nonlinear operators, $\aleph_1$ and $\aleph_2,$ and function $g$ into two functions, $g_1$ and $g_2,$ we are able to get
$$ \aleph_1(\mathscr{L}(x)) - g_1(x) + \aleph_2(\mathscr{L}(x)) - g_2(x)=0.$$
The MHP technique is then introduced as
\begin{equation}\label{eq78}
	\aleph(\vartheta(x),p) = \aleph_1(\vartheta(x)) - g_1(x) + p[\aleph_2(\vartheta(x)) - g_2(x)] = 0,~~~p\in [0,1],
\end{equation}
where mapping $\vartheta$ is approximation of mapping $\mathscr{L}$  and $p$ is an embedding parameter with values ranging from $0$ to $1,$ then $ \aleph_1(\vartheta(x)) = g_1(x)$ to $ \aleph(\vartheta(x)) = g(x).$ This suggests that $ p = 1$ is used to approximate the result of (\ref{eq77}). \\
Let the above solution be in the form of a series,
\begin{equation}\label{eq79}
	\mathscr{L}(x) \simeq \vartheta(x) = \sum_{k=0}^{\infty} p^{k} \vartheta_{k}(x)
\end{equation}
and
\begin{equation}\label{MH7}
	\mathscr{L}(x) = \lim_{p\to1} \vartheta(x).
\end{equation}
\noindent Regarding $(\ref{eq76})$ operators $\aleph_1,\aleph_2$ and function $g$ are given as,

\begin{eqnarray}\label{MH71}
	\left\{
	\begin{array}{ll}
		&\aleph_1 (\mathscr{L}(x))=\mathscr{L}(x),\\
		&\aleph_2(\mathscr{L}(x))=-\frac{9(10+x^{2})}{20(9+x^{2})\Gamma_{1/3}(2/3)}
		\int_{1}^{x}\frac{\left(x^{1/3}-t^{1/3} \right)}{t^{2/3}}\mathscr{L}(t)dt,\\
		&g(x) =\frac{1}{9}x^{-2}.
	\end{array} \right.
\end{eqnarray}

\noindent 
Using (\ref{eq78}) to (\ref{MH71}), we may deduce that
\begin{eqnarray}\label{MH72}
	\begin{aligned}
		&\Big(\sum_{k=0}^{\infty} p^{k} \vartheta_{k}(x)-g_{1}(x)\Big)\\
		&+p\Big(-\frac{9(10+x^{2})}{20(9+x^{2})\Gamma_{1/3}(2/3)}
		\sum_{k=0}^{\infty}p^{k}
		\int_{1}^{x}\frac{\left(x^{1/3}-t^{1/3} \right)}{t^{2/3}} \vartheta_{k}(t)dt-g_{2}(x)\Big)=0.
	\end{aligned}
\end{eqnarray}
\noindent 
We drive an iterative algorithm by rearranging $(\ref{MH72})$ in terms of $ p $ powers as follows\\
Now, \textbf{Algorithm:}
\begin{eqnarray}\label{MH73}
	\left\{
	\begin{array}{ll}
		&\vartheta_0(x) = g_{1}(x),\\
		&\\
		&\vartheta_1(x) =\frac{9(10+x^{2})}{20(9+x^{2})\Gamma_{1/3}(2/3)}
		\int_{1}^{x}\frac{\left(x^{1/3}-t^{1/3} \right)}{t^{2/3}} \vartheta_{0}(t)dt+g(x)-g_{1}(x),\\
		&~~~~~~~~~~~~~~~~~~~~~~~~~~~~~~~~~~~~~~~~~~~~~~~~~~~~~~~~~~~~~~~\\
		&\vartheta_k(\varsigma) =\frac{9(10+x^{2})}{20(9+x^{2})\Gamma_{1/3}(2/3)}
		\int_{1}^{x}\frac{\left(x^{1/3}-t^{1/3} \right)}{t^{2/3}} \vartheta_{k-1}(t)dt,~~k = 2,3....
	\end{array}\right.
\end{eqnarray}
Similar to what is presented in \cite{MHP1,MHP2}, the above algorithm's convergent discussion resembles that of \cite{MHP1,MHP2}.
We use algorithm $(\ref{MH73})$ with initial conditions $\vartheta_0(x)=g_{1}(x)=0$ and $g(x)=\frac{1}{9}x^{-2}$~, or a portion of $g(x)$ to approximate the solution of equation (\ref{eq76}). Thus, the terms in the sequence $\{\vartheta_0(x), \vartheta_1(x),..\}$ can be calculated as follows.\\
\textbf{Case.1,} 
\begin{eqnarray}\label{MH74}
	\begin{array}{ll}
		&\vartheta_0(x)=g_{1}(x)=0,\\
		&\vartheta_1(x)=\frac{9(10+x^{2})}{20(9+x^{2})\Gamma_{1/3}(2/3)}
		\int_{1}^{x}\frac{\left(x^{1/3}-t^{1/3} \right)}{t^{2/3}} \vartheta_{0}(t)dt+\frac{1}{9}x^{-2}-g_{1}(x)=\frac{1}{9}x^{-2},\\
		&\vartheta_2(x)=\frac{9(10+x^{2})}{20(9+x^{2})\Gamma_{1/3}(2/3)}
		\int_{1}^{x}\frac{\left(x^{1/3}-t^{1/3} \right)}{t^{2/3}} \vartheta_{1}(t)dt=\frac{3(10+x^{2})}{400(9+x^{2})\Gamma_{1/3}(2/3)}(-5+\frac{1}{x^{4/3}}+4x^{1/3}),\\
		&\vartheta_k(x)=0, k\geqslant3\\
	\end{array}   
\end{eqnarray}
\textbf{Case.2,}
\begin{eqnarray}\label{MH75}
	\begin{array}{ll}
		&\vartheta_0(x)=g_{1}(x)=\frac{1}{9}x^{-2},\\
		&\vartheta_1(x)=\frac{9(10+x^{2})}{20(9+x^{2})\Gamma_{1/3}(2/3)}
		\int_{1}^{x}\frac{\left(x^{1/3}-t^{1/3} \right)}{t^{2/3}} \vartheta_{0}(t)dt=\frac{3(10+x^{2})}{400(9+x^{2})\Gamma_{1/3}(2/3)}(-5+\frac{1}{x^{4/3}}+4x^{1/3}),\\
		&\vartheta_k(x)=0, k\geqslant2\\
	\end{array}
\end{eqnarray}
\begin{figure}[h]
	\centering
	\includegraphics[width=10cm,height=6cm]{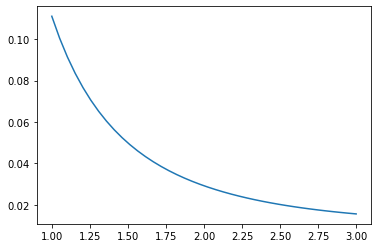}
	\caption{$A\;plot\;showing\;the\;existance\;of\;an\;exact~ solution.$}
	\label{fig:0}
\end{figure}	
\bigskip
\noindent

So, using the equations (\ref{eq79}), (\ref{MH7}), and the findings (\ref{MH74}) and (\ref{MH75}), we get an approximation of the solution to the equation (\ref{eq76}), which is given below :
\begin{equation}\label{MH76}
	\mathscr{L}(x) \simeq \vartheta(x)=\sum_{k=0}^{\infty}\vartheta_{k}(x)
	=\frac{1}{9}x^{-2}+\frac{3(10+x^{2})}{400(9+x^{2})\Gamma_{1/3}(2/3)}(-5+\frac{1}{x^{4/3}}+4x^{1/3}).
\end{equation}
By putting (\ref{MH76}) in (\ref{eq76}) and comparing its both sides, we get absolute point errors to be zero. Therefore, (\ref{MH76}) is the exact solution to (\ref{eq76}). (from figure \ref{fig:0}).  
\begin{rem}
	All computations are done by the Mathematica 11.0 software.
	The plot shown in figure 1 
	is drawn using Python 3.9 and Jupyter Notebook software.
\end{rem}
\vspace{0.5cm}
\textbf{Author contribution} These authors contributed equally.\vspace{0.5cm}\\
\textbf{Funding} This research received no external funding.\vspace{0.5cm}\\
\textbf{Data availability} Not applicable\vspace{0.5cm}\\
\textbf{{\Large Declarations}}\vspace{0.5cm}\\
\textbf{Human and animal ethics} Not applicable\vspace{0.5cm}\\
\textbf{Ethics approval and consent to participate} Not applicable\vspace{0.5cm}\\
\textbf{Consent for publication }All authors give their consent for publication.\vspace{0.5cm}\\
\textbf{Competing interests} The author declares no competing interests.

\end{document}